\documentclass[aos]{imsart}

\RequirePackage{amsthm,amsmath,amsfonts,amssymb}
\RequirePackage[numbers]{natbib}
\RequirePackage[colorlinks,citecolor=blue,urlcolor=blue]{hyperref}%
\RequirePackage{graphicx}%

\startlocaldefs
\theoremstyle{plain}
\newtheorem{prop}{Proposition}
\newtheorem{lem}{Lemma}
\newtheorem{thm}{Theorem}

\theoremstyle{remark}
\newtheorem{defi}{Definition}
\newtheorem{rem}{Remark}
\newtheorem{assu}{Assumption}
\newtheorem{nota}{Notation}
\usepackage{mathrsfs}
\usepackage{bbm}
\usepackage{xcolor}
\usepackage{cleveref}

\newcommand{\EE}{\mathbb{E}}
\newcommand{\NN}{\mathbb{N}}
\newcommand{\PP}{\mathbb{P}}
\newcommand{\RR}{\mathbb{R}}
\newcommand{\VV}{\mathbb{V}}

\newcommand{\cB}{\mathcal{B}}
\newcommand{\cI}{\mathcal{I}}

\newcommand{\cN}{\mathcal{N}}
\newcommand{\cP}{\mathcal{P}}
\newcommand{\cQ}{\mathcal{Q}}

\newcommand{\cX}{\mathcal{X}}
\newcommand{\cY}{\mathcal{Y}}

\newcommand{\sP}{\mathscr{P}}
\newcommand{\sT}{\mathscr{T}}
\newcommand{\sX}{\mathscr{X}}
\newcommand{\sY}{\mathscr{Y}}

\DeclareMathOperator{\Supp}{Supp}

\DeclareMathOperator*{\argmax}{arg\,max}

\let\eps\varepsilon
\let\to\longrightarrow

\newcommand{\indic}{\mathbbm{1}}

\newcommand{\conv}[2][\ ]{\overset{#1}{\underset{#2}{\to}}}
\newcommand{\limtext}[1]{\underset{#1}{\lim}}
\newcommand{\aseq}[2][]{\overset{#1}{\underset{#2}{=}}}

\newcommand{\mbf}[1]{\mathbf{#1}}

\endlocaldefs

\begin{document}

\begin{frontmatter}
\title{Generalized mutual information and their reference priors under Csizar $f$-divergence}%

\begin{aug}
\author{\fnms{Antoine}~\snm{Van Biesbroeck}\ead[label=e1]{antoine.van-biesbroeck@polytechnique.edu}}

\address{CMAP, \'Ecole polytechnique and CEA, Service d'Études M\'ecaniques et Thermiques, France\printead[presep={,\\ }]{e1}}

\end{aug}

\begin{abstract}
{
In Bayesian theory, the role of information is central. The influence exerted by prior information on posterior outcomes often jeopardizes Bayesian studies, due to the potentially subjective nature of the prior choice. In modeling where \emph{a priori} knowledge is lacking, the reference prior theory emerges as a proficient tool. Based on the criterion of mutual information, this theory makes it possible to construct a non-informative prior whose choice can be qualified as objective. In this paper, we contribute to the enrichment of reference prior theory. Indeed, we unveil an original analogy between reference prior theory and Global Sensitivity Analysis, from which we propose a natural generalization of the mutual information definition. Leveraging dissimilarity measures between probability distributions, such as $f$-divergences, we provide a formalized framework for what we term generalized reference priors.
Our main result offers a limit of mutual information, simplifying the definition of reference priors as its maximal arguments. 
This approach opens a new way that facilitates the theoretical derivation of reference priors under constraints or within specific classes. In the absence of constraints, we further prove that the Jeffreys prior maximizes the generalized mutual information considered. 
}

\end{abstract}



\end{frontmatter}

\section{Introduction}

Bayesian theory provides the means for statistical inference over a sampling model through the definition of prior information.
The prior takes the form of a probability distribution that remains to be chosen.
Some easy choices, such as conjugate priors ---when they exist---, 
{allow} efficient computation of \textit{a posteriori} estimates.
However, {as} inferred results are dependent on the chosen prior, 
the subjective nature of 
{such} selection might imperil the trustworthiness of the method.
{Actually}, the notion of information is a statistical concept \cite{Shannon1948} and is thus central in 
Bayesian theory. 
In that sense, Bayes' theorem expresses information transmission from the prior to the observed statistics.

There is consensus that, within a design of experiment declining the benefit of any oracle-sourced knowledge, a proper prior rule belongs within the non informative ones. In this respect,  \citet{Lindley1956} already suggested the maximization of the Shannon's entropy, as a way to choose a prior that would provide the least possible information. %
The same philosophy has been the source of several rules \cite{Kass1996, Datta1996, Berger2008}, which lead to different definitions for non informative and optimal priors. 
{Nevertheless, in practice, their comparison shows that one rule cannot suit all contexts} \cite{Consonni2018, Berger2015, Datta1995, Fuglstad2019}. 
By suggesting maximizing the mutual information, \citet{Bernardo1979} went further 
{since} the resulting non-informative prior 
{must} ---within the model of interest--- favor the knowledge brought by the observations over the prior distribution itself. %
This notion of ``expected utility maximization'' \cite{Beranrdo1979aExp} constitutes what is called reference prior theory, formalized in \cite{Bernardo2005}. The work of \citet{Clarke1994} proves that within its general framework, the reference prior is the one of Jeffreys \cite{Jeffreys1946}, already praised for its property of being invariant by any modeling re-parametrization.

The reference prior theory is nowadays one of the most studied and considered rules to design non informative priors. 
{In practice, the expression of Jeffreys prior is often cumbersome to implement, especially in high-dimensional settings. 
{In such cases, its implementation is therefore not recommended.}
This motivated alternative representations of the reference prior as an limit of priors \cite{Berger2009}.
Also, when the dimension is high, a sequential construction based on a hierarchization of the parameter's space as suggested by \citet{Berger2015} is often favored. This construction consists in deriving what we call a hierarchical version of Jeffreys prior.
On the other hand, 
algorithmic solutions based on variational inference have also been explored to address implementation challenges \cite{Nalisnick2017,GauchyJds2023}.
}

Extensions of the reference prior definition have also been explored (e.g. \cite{Chen2010, Liu2014, LeTriMinh2014}), and some do not necessarily lead to Jeffreys prior \cite{Hashimoto2021, Clarke1997, Gosh2011}. They are all based on the expression of mutual information as an average divergence between the posterior and the prior distributions.

Nevertheless, the reference prior within its original sense has been numerously considered \citep{Berger1996, Clarke1996, Berger2001, Olivera2007}. Its expression within different contexts has been derived and studied, for example, concerning Gaussian process based models \citep{Paulo2005, Gu2018, Mure2019, Mure2021}, Gaussian hierarchical models \citep{Keefe2019}, generalized linear models \citep{Natarajan2000} or general exponential family models \citep{Clarke2010}. Its calculation remains an interest of the literature \citep{Berger2012,Jha2022}, and some works propose a practical implementation of it within a statistical model with simulated \citep{Gu2016,Vollert2019} or even experimental data. In that latter scenario, the considered reference prior can be represented by a Jeffreys prior when the modeling is low-dimensional  \citep{Chen2008, Dandrea2021, VanBiesbroeck2023}, or leverage an approximation approach of a hierarchical Jeffreys prior otherwise \cite{Li2021}, as the one proposed by \citet{Gu2019}.

This paper contributes to the reference prior theory with an original derivation of the mutual information from a Global Sensitivity Analysis (GSA) based viewpoint. %
GSA, whose principle is to measure how the uncertainty of an output is impacted by that of some of its input \citep{Iooss2015}, allows us indeed to provide %
an interpretation of the reference prior as a maximizer of such sensitivity influence that the observations get from the parameter of interest. 
This suggestion leads to defining what we call generalized mutual information, by analogy with global sensitivity indices \citep{DaVeiga2015}. %
It relies on the wide range of existing dissimilarity measures between two probability distributions.
An example is the $f$-divergence subclass \cite{Csiszar1967}, commonly employed as an extension of  Shannon entropy for various purposes in statistics, such as, variational inference \cite{Minka2005,Bach2023}, surrogate model design \cite{Nguyen2009}, PAC Bayesian learning \cite{Picard2022} and Differential Privacy \cite{Mironov2017}.
A study of those divergences within our generalized mutual information is {also} a main contribution of this paper, with the goal of deriving what one is invited to call generalized reference priors.
Pursuing introductory elements we unveiled in \cite{VanBiesbroeckJDS2023}, we go further in this paper with the provision of an accomplished formalization for the generalized reference prior settings, and results based on classical $f$-divergences for sensitivity analysis such as $\alpha$-divergences.
{
Our main result takes the form of {an} limit w.r.t. the number of data of the mutual information. Its analytical expression as a function of the prior permits a simple expression of our reference priors: they are its maximal arguments. It opens the path of easier {theoretical} derivation of reference priors among the ones that satisfy different kinds of constraints, or that belong to some particular classes.}

In the next section, we formalize the usual Bayesian framework that we consider in our work. Our motivation supported by a Global Sensitivity Analysis viewpoint for an enrichment of the mutual information is elucidated in section \ref{sec:Bayesian}. Afterwards, a sub-class of that generalized mutual information is studied in section \ref{sec:motivationGMI} to define and derive what we call their generalized reference priors. 
Eventually, a conclusion of this paper is provided in section \ref{sec:conclusion}.

\newpage

\section{Bayesian framework}\label{sec:Bayesian}

\subsection{Notations and definitions}

In this work, we adopt the notation of conditional kernels for conditional probabilities expressed below. %
Note that their existence and uniqueness is provided by the disintegration theorem (see e.g. \cite{Chang1997}). %
\begin{nota}\label{nota:cond-kernel}
    If $X_1,\,X_2$ are random variables (r.v.) both defined on a probability space $(\Omega,\sP,\PP)$ and taking values respectively in the Borel spaces $(\cX_1,\sX_1)$, $(\cX_2,\sX_2)$, there exists a conditional distribution of $X_1$ to $X_2$: $\PP_{X_1|X_2}:\sX_1\times\cX_2\to[0,1]$ {that} verifies:
        \begin{itemize}
            \item[(i)] $\forall S\in\sX_1$, $x\mapsto\PP_{X_1|X_2=x}(S)$ is measurable;
            \item[(ii)] for $\PP_{X_2}$-a.e. $x$, $S\mapsto\PP_{X_1|X_2=x}(S)$ is a probability measure;
            \item[(iii)] $\forall S\in\sX_1,\,R\in\sX_2,\,\PP(X_1\in S\cap X_2\in R) = \int_{R}\PP_{X_1|X_2=x}(S)d\PP_{X_2}(x)$ where $\PP_{X_2}$ is the distribution of $X_2$.
        \end{itemize}
    It is unique $\PP_{X_2}$-a.s. %
\end{nota}

A statistical model is classically characterized by a collection of parameterized probability measures $(\PP_{Y|T=\theta})_{\theta\in\Theta}$ defined on a measurable space $(\cY ,\sY)$. In practice, $k$ independent realizations of the r.v. $Y$ distributed according to one of the parameterized distributions are observed. Those are denoted $\mbf Y=(Y_1,\dots,Y_k)$, it is a r.v. taking values in $\cY^k$. %

Under the Bayesian viewpoint, $T$ is a random variable taking values in a measurable space $(\Theta,\sT)$ following a prior distribution denoted $\pi$. 
Suppose $T$ and $\mbf Y$ to satisfy Notation \ref{nota:cond-kernel}'s assumptions.
Therefore, the distribution of $\mbf Y$ conditionally to $T$ is $\PP_{\mbf Y|T}=\PP_{Y|T}^{\otimes k}$. %
From those can be respectively defined the posterior distribution $\PP_{T|\mbf Y}$ and the marginal distribution $\PP_{\mbf Y}$ as the a.s. unique ones verifying
    \begin{align}
        &\forall A\in {\sY^{\otimes k}},\, \PP_{\mbf Y}(A) = \int_{\Theta}\PP_{\mbf Y|T=\theta}(A)d\pi(\theta), \\
        & \forall B\in\sT,\, A\in\sY^{\otimes k},\, \int_A\PP_{T|\mbf Y=\mbf y}(B)d\PP_{\mbf Y}(\mbf y) = \int_B \PP_{\mbf Y|T=\theta}(A)d\pi(\theta).\label{eq:formalposterior}
    \end{align}

\subsection{Mutual information}

Within the framework described above, the well-known mutual information \cite{Bernardo1979} issued by the prior $\pi$ and under $k$ observations would be defined as :
\begin{equation}\label{eq:mutaulinforamtiondef}
    I(\pi|k) = \int_{\cY^k}K\!L(\PP_{T|\mbf y}||\pi)d\PP_{\mbf Y}(\mbf y) ,
\end{equation}
with $\PP_{T|\mbf y}$ being a short notation for $\PP_{T|\mbf Y=\mbf y}$, and $K\!L$ denoting the Kullback-Leibler divergence. %
This quantity, built as the measure of the divergence between the posterior and the prior, is interpretable as an indicator of the capacity that the prior has to learn from the observations to issue the posterior.

It is a common and convenient case study to suppose that the model admits a likelihood: there exist density functions $(\ell(\cdot|\theta))_{\theta\in\Theta}$ with respect to  a common measure $\mu$ on $\cY$ such that 
    \begin{equation}
        \forall A\in\sY,\, \PP_{Y|\theta}(A)=\int_{A}\ell(y|\theta)d\mu(y)\ \text{for $\pi$-a.e. $\theta$,}
    \end{equation}
with $\PP_{Y|\theta}$ being a short notation for $\PP_{Y|T=\theta}$. This way $\PP_{\mbf Y}$ and $\PP_{T|\mbf Y}$ respectively admit the densities $p_{\mbf Y}$ with respect to  $\mu^{\otimes k}$ and $p(\cdot|\mbf y)$ with respect to  $\pi$ defined as
    \begin{align}
        &\forall\mbf y\in\cY^k,\,p_{\mbf Y}(\mbf y) = \int_\Theta\prod_{i=1}^k\ell(y_i|\theta)d\pi(\theta) = \int_\Theta\ell_k(\mbf y|\theta)d\pi(\theta),\\
        &\forall\theta\in\Theta,\,\mbf y\in\cY^k,\,p(\theta|\mbf y) = \frac{\ell_k(\mbf y|\theta)}{p_{\mbf Y}(\mbf y)}\ \text{if $p_{\mbf Y}(\mbf y)\ne0$, $p(\theta|\mbf y)=1$ otherwise}.
    \end{align}
This allows us to more conveniently manipulate the mutual information. Also, 
applying Fubini-Lebesgue's theorem gives %
    \begin{equation}\label{eq:mutual-info-fub}
        I(\pi|k) = \int_{\cY^k}K\!L(\PP_{T|\mbf y}||\pi)d\PP_{\mbf Y}(\mbf y) = \int_\Theta K\!L(\PP_{\mbf Y|\theta}||\PP_{\mbf Y})d\pi(\theta).
    \end{equation}
This equality still stands when the integrals above are not finite.
That equation allows an interpretation of the mutual information from a different viewpoint considering the information that the distribution of the observed items would get from the knowledge of the parameter $\theta$. In that sense, as we develop in the next section, this definition of the mutual information may be seen as a global sensitivity index.

\section{Motivation for generalized mutual information}\label{sec:motivationGMI} %

\subsection{A review for the global sensitivity analysis}\label{sec:GSAreview}

Global Sensitivity Analysis constitutes an essential tool for computer experiment based studies.
Whether the point is to verify, to understand or to simplify the system modeling, it consists in quantifying how the uncertainty of the output of a model is influenced by the uncertainty of its inputs \cite{Iooss2015}. %
Since the early work of Sobol' \cite{Sobol1993}, GSA provides a different range of tools to measure statistically how an observed system, denoted $Y$ and supposed to be a function of input variables $Y=\eta(X_1,\dots,X_p)$, is impacted by one or some of the $X_i$ \cite{DaVeiga2021}. In classical GSA settings, the $X_i$'s variables are supposed to follow a known distribution and to be mutually independent.
A general class of such measuring tools is the one of the global sensitivity indices \cite{DaVeiga2015}: considering a dissimilarity measure $D$ between probability distributions, the impact the input $X_i$ has over $Y$ is expressed by
    \begin{equation}\label{eq:GSAindices}
        S_i = \EE_{X_i}[D(\PP_{Y}||\PP_{Y|X_i})] .
    \end{equation}
The choice of the dissimilarity measure $D$ remains open, allowing to cover a wide range of quantities for sensitivity analysis; it can depend on the different expected properties. As an example, setting $D(P||Q)=(\EE_{X\sim P}[X]-\EE_{X'\sim Q}[X'])^2$ gives the un-normalized Sobol' index \cite{Sobol1993}. %

In this work, we focus on the $f$-divergence sub-class of dissimilarity measures \cite{Csiszar1967} expressed as
\begin{equation}
    D_f(\PP_Y||\PP_{Y|X_i}) = \int f\left(\frac{p_Y(y)}{p_{Y|X_i}(y)}\right)p_{Y|X_i}(y)d\mu(y) ,
\end{equation}
where $f:\RR\to\RR$  is usually chosen to be a convex function mapping $1$ to $0$, and $p_Y$, $p_{Y|X_i}$ denoting respectively probability density functions of the distributions of $Y$ and $Y|X_i$ with respect to the reference measure $\mu$. Notice that setting $f=-\log$ leads to $D_f$ being the Kullback-Leibler divergence: $D_{-\log}(P||Q)=K\!L(Q||P)$ for $P,\,Q$ some probability measures.  

Recall now the Bayesian system we introduced in section \ref{sec:Bayesian} and remind the mutual information expression as stated in equation (\ref{eq:mutual-info-fub}):
    \begin{equation}\label{eq:IindexKL}
        I(\pi|k) = \EE_{T\sim\pi}[K\!L(\PP_{\mbf Y|T}||\PP_{\mbf Y})].
    \end{equation}
One can notice that this quantity can be interpreted from the viewpoint of Global Sensitivity Analysis as an index of the form of equation (\ref{eq:GSAindices}), where the Kullback-Leibler divergence is considered as the dissimilarity measure.
{This interpretation is detailed in the following section.}

\subsection{A natural generalization of mutual information}

{
The representation of  mutual information as a global sensitivity index, as stated in section \ref{sec:GSAreview} expands its interpretation.
Under that scope,  our Bayesian study can be conceptualized as if $\mbf Y$ represents the vectorial output of a system, assumed to be a function of $T$, and an additional unknown stochastic variable $\eps$: $\mbf Y=\nu(T,\eps)$. A well-designed Bayesian modeling parameterized by $T$ would be such that the input whose impact on the output is the highest is $T$. This impact is quantified by the mutual information as expressed in equation (\ref{eq:IindexKL}). %
This perspective allows for comparisons of mutual information quantitites derived under different priors, assessing their efficacy in facilitating more objective information transmission compared to others.
}

Additionally, global sensitivity analysis is built from the consideration and the study of different dissimilarity measures, going beyond the Kullback-Leibler divergence. Analogously, we suggest a natural generalization of the mutual information definition:

\begin{defi}[Generalized mutual information]
    Consider the Bayesian framework introduced in section \ref{sec:Bayesian}.
    Consider a dissimilarity measure $D$. The $D$-mutual information of a prior $\pi$ under $k$ observations is 
        \begin{equation}
            I_D(\pi|k) = \EE_{T\sim\pi}[D(\PP_{\mbf Y}||\PP_{\mbf Y|T})].
            \label{eq:defId}
        \end{equation}
\end{defi}

Although there are other suggestions for such general mutual information, their constructions generally start from its first expression, where the divergence considered is between the posterior and prior distributions (equation \ref{eq:mutaulinforamtiondef}). 
{The generalization of this view, supported by the framework of global sensitivity indices, is a contribution of ours.}

The following definition results from the use of $f$-divergences as dissimilarity measures. They constitute the class of generalized mutual information we focus on within this paper.

\begin{defi}[$f$-divergence mutual information]\label{def:Idf}
    The $f$-divergence-mutual information of $\pi$ is defined as 
        \begin{equation}
            I_{D_f}(\pi|k) =\EE_{T\sim\pi}[D_f(\PP_{\mbf Y}||\PP_{\mbf Y|T})] = \int_\Theta\int_{\cY^k}f\left( \frac{p_{\mbf Y}(\mbf y)}{\ell_k(\mbf y|\theta)} \right) \ell_k(\mbf y|\theta)d\mu^{\otimes k}(\mbf y) d\pi(\theta).
        \end{equation}
\end{defi}

Recall that the original mutual information is a particular case of the above definition with $D_f=K\!L$ and $f=-\log$. %

\section{Generalized reference prior under $f$-divergence}\label{sec:genRefP} %

\subsection{Definitions and model assumptions}

The reference prior in the sense of \citet{Bernardo1979} is %
an asymptotic maximal point of the mutual information defined in equation (\ref{eq:mutaulinforamtiondef}). Using the formalization built in \cite{Bernardo2005} for such a notion of asymptotic maximization as a reference, we suggest the following definition for what we call generalized reference priors, i.e. optimal priors within the consideration of our generalized mutual information.

\begin{defi}[Generalized reference prior]\label{def:genref}
    Let $D$ be a dissimilarity measure and $\cP$ a class of priors on $\Theta$. A prior $\pi^\ast\in\cP$ is called a $D$-reference prior over class $\cP$ with rate $\varphi(k)$ if there exists a sequence of compact subsets $(\Theta_i)_{i\in I}$ with $\pi^\ast(\Theta_i)>0$, $\Theta_i\in\Theta$, $\bigcup_{i\in I}\Theta_i=\Theta$ such that
        \begin{equation}\label{eq:defrefpriorsi}
            \forall\pi\in\cP_i,\,\forall i\in I,\, \lim_{k\rightarrow\infty}\varphi(k)(I_D(\pi_i^\ast|k)-I_D(\pi_i|k)) \geq0
        \end{equation}
    where $\cP_i=\{\pi\in\cP,\,\pi(\Theta_i)>0\}$; $\pi_i,\,\pi^\ast_i$ denote the renormalized restrictions to $\Theta_i$ of priors $\pi$ and $\pi^\ast$ respectively;  $\varphi(k)$ is a {positive and}  monotonous function of $k$.
\end{defi}

\begin{rem}
The terminology ``a class of prior" for $\cP$ can seem vague
as the definition of a prior often does not limit itself to probability measures.
Typically, $\sT$ is generally ---and in this work as well--- the Borel $\sigma$-algebra of $\Theta$, and a class of priors over $\Theta$ points out a set of non-negative Radon measures on $\sT$. This way, the term ``renormalized restriction" of $\pi\in\cP$ to $\tilde\Theta$ for some class of priors $\cP$ and some compact subset $\tilde\Theta$ of $\Theta$ has a sense and defines a probability measure.

This way, while we allow our class of priors to include non-informative improper priors when $\Theta$ is not compact, we evaluate the mutual information only to probability measures on compact sets. Indeed $I_D(\pi_i|k)$ from equation (\ref{eq:defrefpriorsi}) is defined as
    \begin{equation}
        I_D(\pi_i|k)= \int_{\Theta_i}D(\PP_{\mbf Y}||\PP_{\mbf Y|T=\theta})d\pi_i(\theta).
    \end{equation}
\end{rem}

This definition matches %
with the literature when the dissimilarity measure is the Kullback-Leibler divergence \cite{Bernardo2005,Mure2018}. Nevertheless, it introduces the definition of an associated rate proposed by ourselves. In fact, our work provides elements showing that such rate exists and may vary as a function of the dissimilarity measure considered. {When it resorts to the original mutual information, the rate is constant.}

The following assumption, which concerns the statistical model, allows us to use fundamental tools for the Bayesian study performed in this work.

\begin{assu}\label{assu:lgolikelihood}
$\Theta$ is an open subset of $\RR^d$ and for $\mu$-almost every $y\in\cY$, $\theta\mapsto\log\ell(y|\theta)$ is twice continuously differentiable with $\nabla^2_\theta\log\ell(y|\theta)$ being negative definite.    
\end{assu}

Under this assumption, the Fisher information matrix defined by $\cI(\theta)=(\cI(\theta)_{i,j})_{i,j=1}^d$ with
    \begin{equation}\label{eq:fisherexpression}
        \cI(\theta)_{i,j}= - \int_\cY[\partial^2_{\theta_i\theta_j}\log\ell(y|\theta)]\, \ell(y|\theta)d\mu(y)
    \end{equation}
is well defined and positive definite. It lets the definition of the Jeffreys prior licit as well; it is the prior whose density with respect to  the Lebesgue measure is $J$ that is proportional to the square root of the determinant of $\cI$: $J(\theta)\propto|\cI(\theta)|^{1/2}$. In this paper, while it is not ambiguous, we make no distinction between the Jeffreys prior and its density.%

An additional assumption ensures ---under a more regular likelihood--- the satisfaction of classical and convenient properties regarding the Fisher information matrix $\cI$
(see e.g. \cite[\S 7.5]{Lehmann1999}). %
To this statement, we denote from now on by  $\|\cdot\|$ the Euclidean norm when applied to a vector in $\RR^d$ and the associated operator norm when applied to a matrix in $\RR^{d\times d}$ (i.e. the largest singular value of the matrix). %
\begin{assu}\label{assu:fisher}
    For any $\hat\theta\in\Theta$, there exists $\delta>0$ such that the quantities
    \begin{equation}
        \EE_\theta\left[\sup_{\theta',\,\|\theta-\theta'\|<\delta} \|\nabla^2_\theta\log\ell(y|\theta')\| \right]\quad\text{and}\quad \EE_\theta\left[\sup_{\theta',\,\|\theta-\theta'\|<\delta}\|\nabla_\theta\log\ell(y|\theta')\|^2\right]
    \end{equation}
    are bounded on the Euclidean ball $B(\hat\theta,\delta)$. 
\end{assu}

Under such assumptions, the asymptotic derivation of the $K\!L$-mutual information provided in \cite{Clarke1994} proves that the Jeffreys prior is the $K\!L$-reference prior over the class of continuous priors (i.e. priors with continuous density with respect to  the Lebesgue measure). In the sense of our definition, its rate is at least constant.

\subsection{Results}

This section focuses on the $f$-divergence-mutual information, for some class of functions $f$ which will be detailed hereafter. It is dedicated to the provision of theoretical results that lead to an expression for $D_f$-reference prior.
First of all let us define a collection of general probability distributions conditional to $T=\theta$, to ease the calculations, and let us define the class of priors we take into account in this work.

\begin{nota}
    The Kolmogorov extension theorem ensures that for any $\theta\in\Theta$, there exists a probability space $(\Omega_\theta,\sP_\theta,\PP_\theta)$ such that for any $k$, and any $B_1,\dots,B_k\in\sY^k$, $\PP_\theta(\mbf Y\in B_1\times\dots B_k)=\PP_{\mbf Y|T=\theta}(B_1\times\dots\times B_k)$.
    We call $\EE_\theta$ the associated expectation.
\end{nota}
\begin{nota}
    Under Assumption \ref{assu:lgolikelihood}, $\Theta$ is an open subset of $\RR^d$. We consider in this work a compact subset $\tilde\Theta$ of $\Theta$. We denote by $\cP$ the class of non-negative Radon measures on $\Theta$ admitting a continuous density with respect to  the Lebesgue measure. $\tilde\cP$ is the class of probability distributions on $\tilde\Theta$ admitting a continuous density with respect to  the Lebesgue measure. To simplify, if $\pi$ is a prior in $\cP$, its density with respect to  the Lebesgue measure is denoted by $\pi$ as well.
\end{nota}

We invite the reader to remark that the upcoming proposition relies on the asymptotic behavior of the ratio $p_{\mbf Y}(\mbf y)/\ell_k(\mbf y|\theta)$ which becomes close to $0$ when $k$ becomes large.
The next assumption formalizes the behavior of the function $f$ in the neighborhood of $0$.

\begin{assu}\label{assu:g1g2}
    $f:(0,+\infty) \to\RR$ {is measurable} and has an asymptotic expansion in the neighborhood of $0$:
    \begin{equation}
        f(x) \aseq{x\rightarrow0^+} g_1(x) + o(g_2(x)),
    \end{equation}
    {for some measurable functions $g_1:(0,+\infty)\to\RR$ and $g_2:(0,+\infty)\to (0,+\infty)$.}
\end{assu}

A first convergence is stated in the proposition below.

\begin{prop}\label{prop:cvproba}
    Suppose Assumptions \labelcref{assu:lgolikelihood,assu:fisher,assu:g1g2}. %
    Consider a prior $\pi\in\tilde\cP$ on $\tilde\Theta$.
    For any $\theta$ in the interior of $\Supp\pi$ there exists $c=c(\theta),\,C$ positive constants such that the following convergence in probability holds:
    \begin{multline}\label{eq:cvinproba}
        \tilde g_2^{C,c}(k^{-d/2})^{-1}\left| f\left(\frac{p_{\mbf Y}(\mbf y)}{\ell_k(\mbf y|\theta)}\right) \right.\\            \left. - g_1\left(k^{-d/2}\pi(\theta)(2\pi)^{d/2}|\cI(\theta)|^{-1/2}\exp\left(\frac{1}{2}S_k^T\cI(\theta)^{-1}S_k \right)\right) \right| %
                \conv[\PP_\theta]{k\rightarrow\infty} 0,
    \end{multline}
    where $\tilde g_2^{C,c}(k^{-d/2})=\sup_{x\in[c,C]} g_2(xk^{-d/2})$, and $S_k$ denotes $\frac{1}{\sqrt{k}}\sum_{i=1}^k\nabla_\theta\log\ell(y_i|\theta)$.
\end{prop}

This first asymptotic result, %
which holds under weak assumptions, 
emphasizes the asymptotic link {between two ratios: (i) the marginal over the likelihood on the one hand, (ii) the chosen prior over Jeffreys one on the other hand.}  
Indeed, as the variables within the exponential in equation (\ref{eq:cvinproba}) are asymptotically standard normal, that term might intuitively become independent of the prior. Therefore, only the ratio between $\pi$ and the Jeffreys prior would remain.
We discuss hereafter explicit formulations of the theorem with different examples of function $f$.

\begin{rem}
    Set $f=-\log$. $g_1$ can be chosen as equal to $f$ and $g_2=1$. In that case we obtain the convergence in $\PP_\theta$-probability already stated in \cite{Clarke1990}:
    \begin{equation}
        \log\frac{\ell_k(\mbf y|\theta)}{p_{\mbf Y}(\mbf y)} - \frac{d}{2}\log k + \log\pi(\theta) - \frac{1}{2}\log|\cI(\theta)| +\frac{1}{2}S_k^T\cI(\theta)^{-1}S_k  \conv[\PP_\theta]{k\rightarrow\infty} 0.
    \end{equation}
\end{rem}

\begin{rem}
    Suppose $f$ is such that $g_1(x)=\alpha g_2(x)$ with $\alpha$ being a nonzero constant and $g_2(x)=x^\beta$, $\beta>0$. The proposition gives
    \begin{multline}
        k^{d\beta/2}\left|f\left(\frac{p_{\mbf Y}(\mbf y)}{\ell_k(\mbf y|\theta)}\right)  - \alpha k^{-d\beta/2}\pi(\theta)^\beta(2\pi)^{d\beta/2}|\cI(\theta)|^{-\beta/2}\exp\left(\frac{\beta}{2} S_k^T\cI(\theta)^{-1}S_k \right)  \right|\\
            \conv[\PP_\theta]{k\rightarrow\infty}0.
    \end{multline}
\end{rem}

To go further on the asymptotic of $I_{D_f}(\pi|k)$, let us consider the following regularity assumptions that we require for the establishment of Theorem \ref{thm:refcompact} expressed right after. %

\begin{assu}\label{assu:infeighes}
    There exists $m>0$ such that $\inf_{\theta'\in\tilde\Theta}-x^T\nabla_\theta^2\log\ell(y|\theta')x>m\|x\|^2\,\forall x$ $\PP_\theta$-a.s. for any $\theta\in\tilde\Theta$.
\end{assu}

\begin{assu}\label{assu:gausstailSk}
     The random variables
     \begin{equation}\label{eq:defSk}
        S_k=\frac{1}{\sqrt{k}}\sum_{i=1}^k\nabla_\theta\ell(y_i|\theta)    
     \end{equation}
     have  Gaussian tails: there exist $\xi,K_1>0$ independent of $\theta$ and $k$ such that $\EE_\theta e^{\frac{\xi}{m}\|S_k\|^2}<K_1$ for any $\theta\in\tilde\Theta$ and $k$.
\end{assu}
It is clear that under the above assumption, the r.v. $X_k=\cI(\theta)^{-1/2}S_k$ admits a similar Gaussian tail with constant $\xi$.
Taking advantage of sub-Gaussian variables theory \cite{Vershynin2018}, the next proposition provides a more convenient and explicit condition to assess the  previous assumption.

\begin{prop}\label{prop:sub-gauss}
    Suppose Assumption \ref{assu:lgolikelihood}. If $\nabla_\theta\log\ell(y|\theta)$ admits a Gaussian tail such that for some $\sigma>0$, $\EE_\theta e^{\sigma\|\nabla_\theta\log\ell(y|\theta)\|^2} <2$ for any $\theta\in\tilde\Theta$, then Assumption \ref{assu:gausstailSk} is verified.
\end{prop}

The last assumption is the following one. It requires some control on $f$.

\begin{assu}\label{assu:finfty}
    $f$ is locally bounded on $(0,+\infty)$ and $f(x)\aseq{x\rightarrow +\infty}O(x)$. %
\end{assu}

\begin{thm}\label{thm:refcompact}
    Suppose Assumptions %
    \labelcref{assu:lgolikelihood,assu:fisher,assu:g1g2,assu:infeighes}.
    Suppose as well Assumption \ref{assu:finfty} with $g_1=\alpha g_2$, $g_2(x)= x^\beta$, $\alpha\in\RR$, $\beta\in(0,1)$, and Assumption \ref{assu:gausstailSk} with $\xi>\beta/2$. 
    Consider a prior $\pi\in\tilde\cP$ on the compact subset $\tilde\Theta$ of $\Theta$. $k^{d\beta/2}I_{D_f}(\pi|k)$ admits a finite limit when $k\to\infty$:
    \begin{equation}\label{eq:liml}
        \lim_{k\rightarrow\infty} k^{d\beta/2} I_{D_f}(\pi|k) = l(\pi) =
\alpha C_\beta \int_{\tilde\Theta}\pi(\theta)^{1+\beta} |\cI(\theta)|^{-\beta/2}  d\theta ,
    \end{equation}
where $ C_\beta = (2\pi)^{d\beta/2} (1-\beta)^{-d/2}$ and where $I_{D_f}(\pi|k)=\int_{\tilde\Theta}D_f(\PP_{\mbf Y}||\PP_{\mbf Y|T=\theta})\pi(\theta)d\theta$.
\end{thm}

Theorem~\ref{thm:refcompact} expresses an limit w.r.t. $k$ of the $D_f$-generalized mutual information when the parameter's space is compact. This result is central for the derivation of reference priors within our frameworks. Indeed, to satisfy the $D_f$-reference prior definition (Definition~\ref{def:genref}), a prior only requires to maximize this limit on its class. In other words, for $\pi_1$ and $\pi_2$ two priors that belong to a certain class of priors, $\pi_1$ induces a more objective modeling than $\pi_2$ if $l(\pi_1)>l(\pi_2)$. We summarize the impact of this result in the remark below.

\begin{rem}
    Under Theorem~\ref{thm:refcompact}'s assumptions, a prior $\pi^\ast$ on $\tilde\Theta$ is a $D_f$-reference prior over class $\cQ$ in the sense of Definition~\ref{def:genref} if and only if it maximizes $l$ over $\cQ$, with $l$ defined in equation (\ref{eq:liml}). Additionally, if $\cQ$ is a convex class, then $\pi^\ast$ is unique.
\end{rem}

When the class $\tilde\cP$ {is not constrained}, for instance, this maximal argument is Jeffreys prior, as the following proposition indicates.

\begin{prop}\label{prop:jeffcompact}
Under the assumptions of Theorem~\ref{thm:refcompact}, if $\alpha<0$, then
    \begin{equation}\label{eq:thmI-Jgeqpi}
        \lim_{k\rightarrow\infty}k^{d\beta/2}(I_{D_f}(J|k)-I_{D_f}(\pi|k))\geq 0 ,
    \end{equation}
where $J(\theta)=|\cI(\theta)|^{1/2}/\int_{\tilde\Theta}|\cI(\tilde\theta)|^{1/2}d\tilde\theta$ denotes the Jeffreys prior. The equality stands if and only if $\pi=J$.
\end{prop}

\begin{rem}\label{rem:overlinef}
    Under the assumptions of Theorem~\ref{thm:refcompact}, if we consider a function $\overline{f}=f+\gamma$ for some real constant $\gamma$, we notice
        \begin{equation}
            I_{D_{\overline{f}}}(\pi|k) = I_{D_f}(\pi|k) + \gamma
        \end{equation}
    for any prior $\pi$. Therefore, the  statement of Proposition~\ref{prop:jeffcompact} (equation (\ref{eq:thmI-Jgeqpi})) still stands in such a case.
\end{rem}

Theorem~\ref{thm:refcompact} describes the asymptotic behavior of the $D_f$-mutual information under compact sets. With little additional work and the consideration of the remark above, we get the following result.

\begin{thm}\label{thm:mainfinal}
    Under the assumptions of Theorem~\ref{thm:refcompact}, and considering that Assumptions \labelcref{assu:infeighes,assu:gausstailSk} hold for every compact subset $\tilde\Theta$ of $\Theta$, the Jeffreys prior is {the $D_{\overline f}$-reference prior over $\cP$ in the sense of Definition \ref{def:genref}, for any $\overline{f}$ defined as $\overline{f}=f+\gamma$ for some real constant $\gamma$.} It has a rate of $k^{d\beta/2}$. 
\end{thm}

{
Theorem~\ref{thm:mainfinal} indicates that our generalized mutual information measures still are, in our context, globally maximized by the Jeffreys prior.}

\subsection{Discussions}

\subsubsection{Discussion of the assumptions}

Theorem \ref{thm:mainfinal} extends the preliminary work published in \cite{VanBiesbroeckJDS2023} and recalled in Appendix \ref{app:betaleq0}.  
In \cite{VanBiesbroeckJDS2023} the statement was restricted to $\beta<0$
while Theorem \ref{thm:mainfinal} considers general $f$-divergences with $f$ behaving like an $x^\beta$, $\beta>0$ around $0$. %
This stronger result presents the benefit of including classical dissimilarity measures used in sensitivity analysis, such as $\alpha$-divergence, expressed as $f$-divergence with 
    \begin{equation}
        f(x) = \frac{x^\alpha-\alpha x-(1-\alpha)}{\alpha(\alpha-1)}\quad,\quad \alpha\in(0,1).
    \end{equation}

The assumptions over the likelihood may remain somehow restrictive. However they might still be verified by regular statistical models like the univariate normal model with known variance, or compact models. {Note that due to the asymptotic normality of the r.v. $X_k$ induced by the Central Limit Theorem, the parameter $\xi$ of its assumed Gaussian tail cannot exceed $1/2$. This restricts $\beta$ to be smaller than $1$.}

Furthermore, we would like to point out to the reader the two following remarks, raised from the proofs of our results.
First, the class of prior $\cP$ we take into account might be enlarged to the class of all priors admitting a Lebesgue density that is continuous a.e. and locally bounded.
Second, in opposition with the classical settings for $f$-divergences which tend to suppose the convexity of $f$, the Jeffreys prior optimality within our results is actually implied by the concavity of $x\mapsto f(1/x)$ in the neighborhood of $0$.
This gives some insight to the use of the Chi-Square distance within the mutual information as suggested in \cite{Clarke1997}. Indeed, this dissimilarity measure corresponds to a particular case of our work when $\beta=-1$, making $x\mapsto f(1/x)$ both convex and concave.
A conclusion about the optimum of such generalized
mutual information thus requires stronger regularity assumptions on the likelihood to push
the asymptotic analysis further. In their work, \citet{Clarke1997} have shown in a particular context that
this could lead to the inverse of Jeffreys prior as the optimum.

A global sketch of our proof giving an heuristic of the phenomena is proposed in Appendix~\ref{app:heuristic}. A rigorous proof is given in the next section.\\

\subsubsection{Discussion of the result}

Regarding the optimal characteristic of the Jeffreys prior, we argue that it actually supports  the relevance of our construction, in continuity with the original reference prior framework which we rely on. Indeed, we remind that Jeffreys prior is optimal for $K\!L$-mutual information \cite{Clarke1994}. Although its multivariate derivation is criticized, its univariate form remains robust and appreciated in low-dimensional cases.

For a practical implementation when the dimension is high, one would be invited to refer to hierarchical strategies such as the one proposed by \citet{Berger2015}. 
As these strategies are based on the sequential derivation of reference priors on sequentially conditioned models, they take the form of hierarchically derived low dimensional Jeffreys priors. Consequently, our Theorem~\ref{thm:mainfinal} also supports their methodology.

Additionally, we draw attention to the open-ended nature of the result we express in Theorem~\ref{thm:refcompact}. Indeed, analytic maximization of our function $l$ over a well-chosen class of priors constitutes an easier path to express reference priors that might differ from Jeffreys. %
The following section suggests an example of a constrained problem.

\subsubsection{Discussion of constraints introduction}

Here we illustrate with a simple example an application of Theorem~\ref{thm:refcompact} to the derivation of constrained reference priors.
Consider a Bernoulli modeling: $Y|(T=\theta)\sim\cB(\theta)$, $\cB(\theta)$ being the Bernoulli distribution with parameter $\theta$. We take $\Theta=[0,1]$ and write $\ell(y|\theta)=\theta^{y}(1-\theta)^{1-y}$.

Within this modeling, the Jeffreys prior takes the form of a Beta distribution: $\mathrm{Beta}(1/2,1/2)$ (see e.g. \cite{Robert2007}).

An intuitive type of constraints is the restriction of the moments of the prior. For instance,
we denote $\cP_\EE$ the class of priors $\cP_\EE=\{\pi=\mathrm{Beta}(\lambda_1,\lambda_2),\,\EE_{T\sim\pi}[T]=1/c\}$. A prior in $\cP_\EE$ can be parameterized as $\pi_\lambda = \mathrm{Beta}(\lambda,\lambda(c-1))$, $\lambda\in(0,1)$.
Then, %
solving $\lambda^\ast=\argmax_\lambda l(\pi_\lambda)$ can be carried out numerically or analytically to issue the reference prior $\pi_{\lambda^\ast}$ over $\cP_\EE$.
The same way, some class $\cP_{\VV}$ could be considered this time to fix the variance of the prior.

In Figure \ref{fig:exbeta} are plotted the $D_f$-reference priors over different classes of Beta priors with constrained expectations first, and with constrained variances second. We notice that, when it concerns the class $\cP_\EE$, two priors are reference priors. Actually, the considered class is not convex in this case, which makes the uniqueness not ensured. An additional constraint on the researched reference prior should be set to make it unique. For instance, with the idea of constructing non-informative priors, the one with maximal Shanon's entropy could be chosen. In our example, $\pi_1^\ast$ (Figure \ref{fig:exbeta}.(a)) has a higher entropy than $\pi_2^\ast$ (Figure \ref{fig:exbeta}.(b)).

\begin{figure}
    \centering%
    \includegraphics[width=0.32\linewidth]{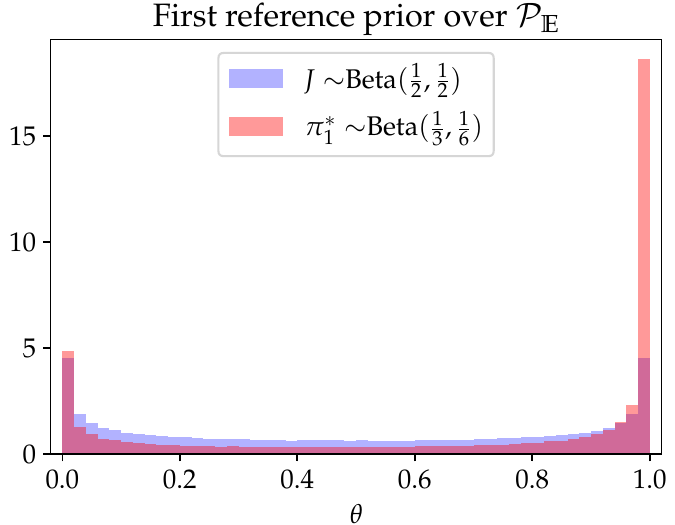}\hspace*{0.01\linewidth}%
    \includegraphics[width=0.32\linewidth]{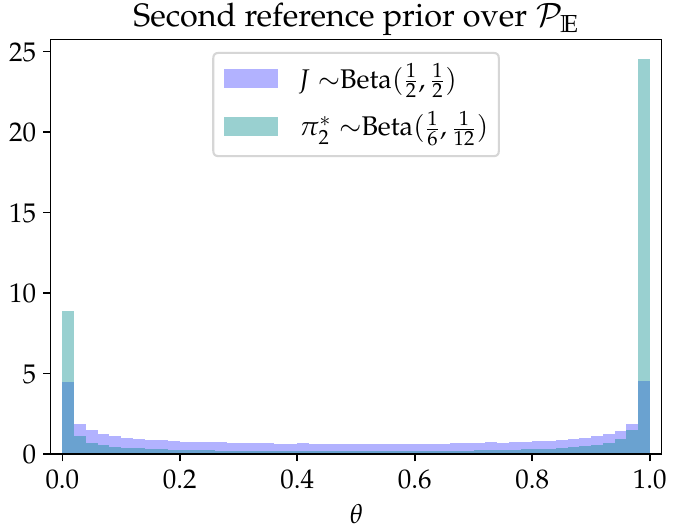}\hspace*{0.01\linewidth}%
    \includegraphics[width=0.32\linewidth]{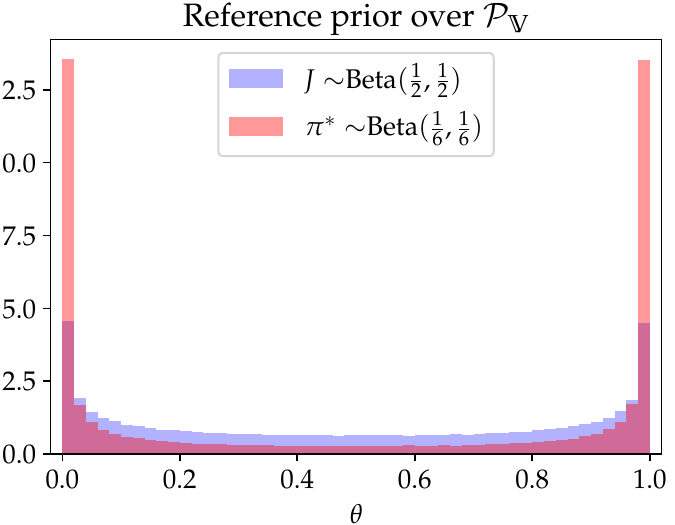}\\
    {\hspace{\stretch{1}}(a)\hspace{\stretch{2}}(b)\hspace{\stretch{2}}(c)\hspace{\stretch{1}}\ }%
    \caption{Comparison of histograms of samples from Jeffreys prior and from the $D_f$-reference priors over a class $\cP_\EE=\{\pi=\mathrm{Beta},\,\EE_{T\sim\pi}T=\frac{2}{3}\}$ (Figures (a) and (b)), and over a class $\cP_\VV=\{\pi=\mathrm{Beta},\,\VV_{T\sim\pi}T=\frac{3}{16}\}$ (Figure (c)). In this example, their exist 2 solutions $\pi_1^\ast$ (in Figure (a)) and $\pi_2^\ast$ (in Figure(b)) of reference priors over $\cP_\EE$. $\beta$ has been set to $1/2$ in this example.}
    \label{fig:exbeta}
\end{figure}

\subsection{Proofs}

\begin{proof}[Proof of Proposition \ref{prop:cvproba}]
    Let $\eps,\tilde\eps>0$ and fix $\theta$ in the interior of $\Supp\pi$. %
First, let us build {four} events which will be asymptotically negligible:
\begin{enumerate}
    \item {Remind the definition of the r.v. $S_k$ given in equation (\ref{eq:defSk}).} The central limit theorem ensures that when $\mbf y\sim\PP_{\mbf Y|T=\theta}$, $S_k$ converges in distribution to $\cN(0,\cI(\theta))$ as $k\to\infty$. Therefore, one can choose $M$ such that the event $A_k$ defined as $A_k=\{\|S_k\|>M\}$ verifies $\PP_\theta(A_k)\conv{k\rightarrow\infty}\eps/5$.
    \item Denote for any $i=1,\dots,k$,
        \begin{equation}
            K_i=\{-x^T\nabla^2_\theta\log\ell(y_i|\theta')x,\,\theta'\in\tilde\Theta,\,\|x\|=1\}.
        \end{equation}
     They are compact subsets of $(0,+\infty)$, thus, the law of large numbers gives 
     \begin{equation}
        \frac{1}{k}\sum_{i=1}^k\inf K_i\conv{k\rightarrow\infty}\EE_\theta\left[\inf_{\theta'\in\tilde\Theta,\,\|x\|=1}-x^T\nabla^2_\theta\log\ell(y_1|\theta')x\right]  ,
     \end{equation}
    which is a positive constant as the infima are positive $\PP_\theta$-a.s. We choose an $\hat m>0$ smaller than this limit and we denote
        \begin{equation}
            B_k=\left\{\frac{1}{k}\sum_{i=1}^k\inf K_i<\hat m\right\}.       
        \end{equation}
     Then $\PP_\theta(B_k)\conv{k\rightarrow\infty}0$.
    \item Denote
        \begin{equation}
            C =\exp\left(M^2/(2\hat m)\right) \quad\text{and}\quad  c = \pi(\theta)(2\pi)^{d/2}|\cI(\theta)|^{-1/2} ,
        \end{equation}
        which are finite and positive. Using the assumed local behavior of $f$ around $0$, one can choose $\nu,\eta>0$ such that for any $c<|x|<C$, $|l|<\nu$, and $|h|<\eta$, 
        \begin{equation}
            |f(xh+lh)-g_1(xh)|\leq\tilde\eps g_2(xh)\leq\tilde\eps\tilde g_2^{C,c}(h).    
        \end{equation}
        Using the $\PP_\theta$-a.s. convergence of $-\frac{1}{k}\sum_{i=1}^k\nabla^2_\theta\log\ell(y_i|\theta)$ to $\cI(\theta)$, we have $\PP_\theta(C_k)\conv{k\rightarrow\infty}0$ where
        \begin{align}
        \nonumber
            C_k = &\left\{\forall \|x\|\leq M,\, \left|\pi(\theta)(2\pi)^{d/2}|k\cI(\theta)|^{-1/2}\exp\left(\frac{1}{2}x^T \cI(\theta)^{-1}x\right) \right.\right. \\
            & \left.\left. - \pi(\theta)(2\pi)^{d/2}|k\hat\cI_k(\theta)|^{-1/2}\exp\left(\frac{1}{2}x^T \hat\cI_k(\theta)^{-1} x\right) \right|\geq\nu k^{-d/2}\right\} ,
        \end{align}
        with $\hat\cI_k(\theta')$ denoting $-\frac{1}{k}\sum_{i=1}^k\nabla^2_\theta\log\ell(y_i|\theta')$.
    \item Let us choose $D>0$ large enough to have 
        \begin{equation}
            \sup_{\tilde\Theta}|\pi| \int_{\|x\|>D} e^{Mx}e^{-\frac{\hat m}{2}x^2} dx<\nu ,    
        \end{equation}
    and $u>0$ such that $\forall|x|<uD^2,\,|e^{x}-1|<\nu/(2\sup_{\tilde\theta}\pi CD^d)$. 
    
    From Assumption \ref{assu:fisher}, the function $\theta'\mapsto\|\nabla_\theta^2\log\ell(y|\theta')-\nabla^2_\theta\log\ell(y|\theta)\|$ is dominated on the Euclidean ball $B(\theta,\delta)$ by a variable that is integrable under $\PP_\theta$. Therefore the function
        \begin{equation}
            \theta'\mapsto \EE_\theta \left[\|\nabla_\theta^2\log\ell(y|\theta')-\nabla^2_\theta\log\ell(y|\theta)\|\right]
        \end{equation}
        converges to $0$ when $\theta'$ converges to $\theta$ as a consequence of the continuity of the second order partial derivatives of the log-likelihood.

        Thus, {if we choose $\tilde\delta$ such that the above quantity is smaller than %
        $u\eps$ and if we define $D_k$ as the event}
        \begin{equation}
            D_k=\left\{\sup_{\theta'\in\tilde\Theta,\,\|\theta'-\theta\|<\tilde\delta}\|\hat\cI_k(\theta')-\hat\cI_k(\theta)\|>u \right\}, %
        \end{equation}
        one gets by applying Markov's inequality that $\PP_\theta(D_k)<\eps/5$ for any $k\geq 1$.
\end{enumerate}

Now, notice that for any $\tilde\theta\in\tilde\Theta$, there exists $\xi(\tilde\theta,\theta)$ on the segment between $\theta$ and $\tilde\theta$ such that
\begin{equation}\label{eq:DLlogfraclik}
    \log\left(\frac{\prod_{i=1}^k \ell(y_i|\tilde\theta)}{\prod_{i=1}^k \ell(y_i|\theta)}\right) = (\tilde\theta-\theta)^T\sum_{i=1}^k\nabla_\theta\log \ell(y_i|\theta) +\frac{1}{2}(\tilde\theta-\theta)^T\sum_{i=1}^k\nabla_\theta^2\log \ell(y_i|\xi(\tilde\theta,\theta))(\tilde\theta-\theta).
\end{equation}
Let us write %
$p_{\mbf Y}(\mbf y)/\ell_k(\mbf y|\theta) = \int_\Theta h_{\mbf y}(\tilde\theta,\theta)\pi(\tilde\theta)d\tilde\theta$ where
\begin{equation}\label{eq:hydef}
    h_{\mbf y}(\tilde\theta,\theta) = \exp\left((\tilde\theta-\theta)^T\nabla_\theta\log\ell_k(\mbf y|\theta)\right)  \exp \left( \frac{1}{2}(\tilde\theta-\theta)^T\nabla^2_\theta\log\ell_k(\mbf y|\xi(\tilde\theta,\theta))(\tilde\theta-\theta) \right).
\end{equation}
We can bound $h_{\mbf y}(\tilde\theta,\theta)$ under the event $A_k^c\cap B_k^c\cap C_k^c\cap D_k^c$ as
    \begin{equation}\label{eq:inegh}
        h_{\mbf y}(\tilde\theta,\theta) \leq\exp\left( M\sqrt{k}\|\tilde\theta-\theta\| \right)\exp\left( -\frac{k}{2}\hat m\|\tilde\theta-\theta\|^2 \right)\leq C.
    \end{equation}

{Still working under $A_k^c\cap B_k^c\cap C_k^c\cap D_k^c$, we bound by $\nu k^{-d/2}$ the four quantities enumerated below:} %
\begin{itemize}
    \item Define $\delta_k=Dk^{-1/2}$. %
    Inequality (\ref{eq:inegh}) implies
        \begin{equation}
            \int_{\|\tilde\theta-\theta\|>\delta_k} h_{\mbf y}(\tilde\theta,\theta)\pi(\tilde\theta)d\tilde\theta <\nu k^{-d/2}.
        \end{equation}
    \item As the same inequality is verified by the integrated term, we have
        \begin{equation}
            \int_{\|\tilde\theta-\theta\|>\delta_k} \exp\left((\tilde\theta-\theta)^T\sqrt{k}S_k\right)\exp\left( -\frac{k}{2}(\tilde\theta-\theta)^T\hat\cI_k(\theta)(\tilde\theta-\theta) \right)\pi(\theta)d\tilde\theta <\nu k^{-d/2}.
        \end{equation}
    \item 
    Recalling the continuity of $\pi$, 
    a large enough $k$  allows to state $|\pi(\tilde\theta)-\pi(\theta)|\leq\nu/(2CD^d)$ for any $\|\tilde\theta-\theta\|<\delta_k$. We write:
        \begin{multline}\label{eq:continuityofHatIk}
           \left|\int_{\|\tilde\theta-\theta\|<\delta_k} h_{\mbf y}(\tilde\theta,\theta)\pi(\tilde\theta)d\tilde\theta \right.\\ 
                \left. - \int_{\|\tilde\theta-\theta\|<\delta_k}\exp\left((\tilde\theta-\theta)^T\sqrt{k}S_k\right)\exp\left(-\frac{k}{2}(\tilde\theta-\theta)^T\hat\cI_k(\theta)(\tilde\theta-\theta)\right)\pi(\theta)d\tilde\theta \right|\\ %
                \leq \int_{\|\tilde\theta-\theta\|<\delta_k}|\hat h_{\mbf y}(\tilde\theta,\theta,\xi(\tilde\theta,\theta)) - \hat h_{\mbf y}(\tilde\theta,\theta,\theta) |\pi(\tilde\theta)d\tilde\theta + \int_{\|\tilde\theta-\theta\|<\delta_k}\hat h_{\mbf y}(\tilde\theta,\theta,\theta)|\pi(\tilde\theta)-\pi(\theta)|d\tilde\theta
        \end{multline}
    where $\hat h(\tilde\theta,\theta,\theta')=\exp((\tilde\theta-\theta)^T\sqrt{k}S_k)\exp\left(-\frac{k}{2}(\tilde\theta-\theta)^T\hat\cI_k(\theta')(\tilde\theta-\theta)\right)$. {Notice that for any $\theta',\tilde\theta$ in the Euclidean ball $B(\theta,\delta_k)$,}
        \begin{align}
        \nonumber
            & |\hat h_{\mbf y}(\tilde\theta,\theta,\theta') - \hat h_{\mbf y}(\tilde\theta,\theta,\theta) | \\
               & \leq \hat h_{\mbf y}(\tilde\theta,\theta,\theta')\left|\exp\left( -\frac{k}{2}(\tilde\theta-\theta)^T(\hat\cI_k(\theta')-\hat\cI_k(\theta))(\tilde\theta-\theta) \right)-1\right|.
        \end{align}
    Thus, under $D_k^c$ and while $\delta_k<\tilde\delta$, we find the quantity of equation (\ref{eq:continuityofHatIk}) to be smaller than $\nu k^{-d/2}$.
    \item By recognizing a Gaussian density in the integral below, we get the equality 
    \begin{align}
    \nonumber
        &\int_{\RR^d}\exp((\tilde\theta-\theta)^T\sqrt{k}S_k)\exp\left(\frac{1}{2}(\tilde\theta-\theta)^T\nabla_\theta^2\ell_k(\mbf y|\theta)(\tilde\theta-\theta)\right)\pi(\theta)d\tilde\theta \\
        &=
        \pi(\theta)(2\pi)^{d/2}|k\hat\cI_k(\theta)|^{-1/2}\exp\left(\frac{1}{2}S_k^T\hat\cI_k(\theta)^{-1}S_k\right).
        \label{eq:intRnabla2bisbis}
    \end{align}
    It implies that under $C_k^c$, 
    \begin{equation}
        \left|(\ref{eq:intRnabla2bisbis}) - \pi(\theta)(2\pi)^{d/2}|k\cI(\theta)|^{-1/2}\exp\left(\frac{1}{2}S_k^T\cI(\theta)^{-1}S_k\right)\right|\leq\nu k^{-d/2}.
    \end{equation}
\end{itemize}

Eventually, {we choose $k$ large enough to have %
$k^{-d/2}<\eta/4$, and $\PP_\theta(A_k\cup B_k\cup C_k\cup D_k)<\eps$}. Combining the four upper-bounds elucidated above, we obtain that with $\PP_\theta$-probability larger than $1-\eps$, the following inequality holds:
\begin{align}
\nonumber  
  \left|f\left(\int_\Theta h_{\mbf y}(\tilde\theta,\theta)\pi(\tilde\theta)d\tilde\theta\right) - g_1\left(k^{-d/2} \pi(\theta)(2\pi)^{d/2}|\cI(\theta)|^{-1/2}\exp\left(\frac{1}{2}S_k^T\cI(\theta)^{-1}S_k\right)\right)   \right| \\
\leq \tilde\eps \tilde g_2^{C,c}(k^{-d/2}) .
    \end{align}
This concludes the proof of the desired convergence in $\PP_\theta$-probability.
\end{proof}

\begin{proof}[Proof of Proposition \ref{prop:sub-gauss}]
Under Assumptions \labelcref{assu:lgolikelihood,assu:fisher}, the expectation $\EE_\theta\left[\partial_{\theta_i} \log\ell(y|\theta)\right]$ is null for any $\theta\in\Theta$ and $i=1,\ldots,d$. Therefore, we deduce from \cite[Proposition 2.5.2]{Vershynin2018} that there exists $c>0$ such that, for any $i=1,\ldots,d$ and  for any $t \in \RR$,
$$
\EE_\theta [\exp (t \partial_{\theta_i} \log\ell(y|\theta))]\leq \exp( c^2t^2).
$$
This way, we also have for any $k>0$,
$\EE_\theta [\exp ( t S_{k,i})] \leq \exp( c^2t^2)$ for any $t\in \RR$, and 
$\EE_\theta [\exp ( t |S_{k,i}|)] \leq 2\exp( c^2t^2)$ for any $t\geq 0$. By Markov's inequality this gives for any $s,t\geq0$:
\begin{align*}
\PP_\theta\big(\|S_k\|\geq s\big)& \leq \sum_{i=1}^d \PP_\theta\big(|S_{k,i}|\geq s/d \big) =
\sum_{i=1}^d \PP_\theta\big(\exp(t |S_{k,i}|)\geq \exp(t s/d) \big) 
\\
&\leq 2 d 
\exp(c^2 t^2 - ts/d).
\end{align*}
This holds for any $t \geq 0$ so we can choose $t=s/(2dc^2)$ and we get
$$
\PP_\theta\big(\|S_k\|\geq s\big) \leq 2d \exp\big(-s^2/(4d^2 c^2)\big).
$$
Eventually, if $\tilde\xi<1/(4d^2c^2)$, then we can derive
\begin{align}
\nonumber  
        \EE_\theta [\exp( \tilde\xi\|S_k\|^2)] = \int_0^\infty\PP_\theta(e^{\tilde\xi\|S_k\|^2}\geq u)du&\leq 1+\int_0^\infty 2\tilde\xi s e^{\tilde\xi s^2}\PP(\|S_k\|\geq s)ds \\
          &\leq 1+\int_0^\infty 4d\tilde\xi se^{s^2\left(\tilde\xi-\frac{1}{4d^2c^2}\right)}ds <\infty.
    \end{align}
Thus, $S_k$ admits a Gaussian tail as claimed. 
\end{proof}

\begin{proof}[Proof of Theorem \ref{thm:mainfinal} and of Proposition~\ref{prop:jeffcompact}]
    To prove this theorem, we will use the {two} following Lemmas which are proved later on.
   \begin{lem}\label{lem:technicalLemma}
  For any $\zeta<2\xi$, there exists a constant $K_2>0$ independent of $\theta\in\tilde\Theta$ such that for any $k>0$,
        \begin{equation}
            k^{d\zeta/2}\EE_\theta\left[\left(\frac{p_{\mbf Y}(\mbf y)}{\ell_k(\mbf y|\theta)}\right)^\zeta\right]\leq K_2.
        \end{equation}
\end{lem}    
\begin{lem}\label{lem:lem2}
    There exists a constant $K_3>0$ independent of $\theta\in\tilde\Theta$ %
    such that for any $k>0$ and any $\rho>0$, %
    \begin{equation}
        \EE_\theta\left[ \indic_{E_k} \frac{p_{\mbf Y}(\mbf y)}{\ell_k(\mbf y|\theta)} \right] \leq \rho^{1/2} K_3 k^{-d\beta/2}
    \end{equation}
    with $E_k=\{ {\mbf y} : \frac{p_{\mbf Y}(\mbf y)}{\ell_k(\mbf y|\theta)}>\rho^{-1/\beta}\}$.
\end{lem}

We first state the equi-integrability of r.v. of equation (\ref{eq:cvinproba}): Let $\rho>0$ and choose $A$ such that $\PP_\theta(A)=\rho$. First, denoting $X_k=\cI(\theta)^{-1/2}S_k$, 
\begin{align}\label{eq:equiintpart1}
    \EE_\theta\left[\indic_A \pi(\theta)^\beta(2\pi)^{d\beta/2}e^{\frac{\beta}{2}S_k^T\cI(\theta)^{-1}S_k}\right]  
        &\leq \pi(\theta)^\beta(2\pi)^{d\beta/2} \rho^{1/p} \EE_\theta\left[e^{\frac{\beta}{2}q\|X_k\|^2}\right]^{1/q} \\
        &\leq \rho^{1/p}(\sup_{\Theta}\pi)^\beta (2\pi)^{d\beta/2}K_1^{1/q}\nonumber
\end{align}
for any $p,q$ such that $1/p+1/q=1$ and $q\beta/2<\xi$. %
Second, consider $\hat C$ and $\hat C'$ such that $|f(x)|\leq\hat Cx$ if $x>\rho^{-1/\beta}$ and $|f(x)|\leq \hat C'x^\beta$ otherwise.
Using $E_k=\{ {\mbf y}:p_{\mbf Y}(\mbf y)/\ell_k(\mbf y|\theta)>\rho^{-1/\beta}\}$, we compute
    \begin{align}
\nonumber      
& \EE_\theta\left[ \indic_Ak^{d\beta/2}f\left( \frac{p_{\mbf Y}(\mbf y)}{\ell_k(\mbf y|\theta)}\right) \right]
    \\
\nonumber                  &\leq \EE_\theta\left[ \indic_{E_k}\indic_Ak^{d\beta/2} \hat C  \frac{p_{\mbf Y}(\mbf y)}{\ell_k(\mbf y|\theta)} \right] + \EE_\theta\left[ \indic_{E_k^c}\indic_Ak^{d\beta/2}\hat C'\left( \frac{p_{\mbf Y}(\mbf y)}{\ell_k(\mbf y|\theta)}\right)^\beta \right] \\
 \label{eq:equiintpart2}
                & \leq \hat C\rho^{1/2} K_3 +  \rho^{1/p} K_2^{1/p} , %
\end{align}
where $1/p=1-1/q$ with $q\beta/2<\xi$. Equations (\ref{eq:equiintpart1}) and (\ref{eq:equiintpart2}) ensure the required equi-integrability. %
That makes the convergence {in $\PP_\theta$-probability stated in Proposition \ref{prop:cvproba}} to stand in $L^1(\PP_\theta)$. %

    Thus, we get as $k\rightarrow\infty$
        \begin{align}
            k^{d\beta/2} D_f(\PP_{\mbf Y} || \PP_{\mbf Y|\theta}) = \alpha\pi(\theta)^\beta(2\pi)^{d\beta/2}|\cI(\theta)|^{-\beta/2}\EE_\theta\left[e^{\frac{\beta}{2} \|X_k\|^2}\right] + o(1).    
        \end{align}
    The sequence of r.v. $X_k$ converges in distribution to a standard normal r.v. denoted $X$. %
    This implies that for any $a>0$, the expectation $\EE_\theta\left[e^{\frac{\beta}{2}\|X_k\|^2}\indic_{\|X_k\|\leq a}\right]$ converges to $\EE_\theta\left[e^{\frac{\beta}{2}\|X\|^2}\indic_{\|X\|\leq a}\right]$. Moreover, under Assumption \ref{assu:gausstailSk} there exists a constant $K'$ independent of $k$ such that for any $k$
        \begin{align}\label{eq:expSkindica}
            \EE_\theta\left[ e^{\frac{\beta}{2}\|X_k\|^2}\left|1-\indic_{\|X_k\|\leq a}\right|\right] 
                & \leq \EE_\theta\left[e^{\frac{\beta q}{2}\|X_k\|^2}\right]^{1/q} \PP_\theta\left(\|X_k\|\geq a \right)^{1/p} \\
\nonumber            &\leq K_1^{1/q} \PP_\theta(\exp(t^2 \| X_k\|^2\geq\exp(t^2a) )\\
\nonumber       & \leq K' e^{-t^2 a^2/p} 
        \end{align}
    using H\"older inequality with $1/p+1/q=1$ such as $\beta q/2\leq \xi $ first, and Markov inequality with $t^2<\xi$ second. Thus, for any $\eps>0$, one can choose $a$ {such that $(\ref{eq:expSkindica}) \leq\eps/3$ and $\EE_\theta\left[e^{\frac{\beta}{2}\|X\|^2}\left|1-\indic_{\|X\|\leq a}\right|\right]<\eps/3$ as well to get
    \begin{align}
    \nonumber
    &   \left|\EE_\theta\left[e^{\frac{\beta}{2}\|X_k\|^2}\right] - \EE_\theta\left[e^{\frac{\beta}{2}\|X\|^2}\right]\right| \\
    & < 2\eps/3 + \left|\EE_\theta\left[e^{\frac{\beta}{2}\|X_k\|^2}\indic_{\|X_k\|\leq a}\right] - \EE_\theta\left[e^{\frac{\beta}{2}\|X\|^2}\indic_{\|X\|\leq a}\right]\right|.
    \end{align}
    Eventually, $k$ can be chosen large enough such that the rightt-hand side of above equation is smaller than $\eps$, which states the convergence $\EE_\theta\left[e^{\frac{\beta}{2}\|X_k\|^2}\right] \conv{k\rightarrow\infty} \EE_\theta\left[e^{\frac{\beta}{2}\|X\|^2}\right]$.} This way, 
        \begin{align}
            k^{d\beta/2} D_f(\PP_{\mbf Y} || \PP_{\mbf Y|\theta}) \aseq{k\rightarrow\infty} \alpha\pi(\theta)^\beta(2\pi)^{d\beta/2}|\cI(\theta)|^{-\beta/2}(1-\beta)^{-d/2} + o(1).    
        \end{align}

    Now, our aim is to dominate this term to integrate it with respect to  $\theta$. We consider $\hat C$ and $\hat C'$ such that $|f(x)|\leq\hat Cx$ if $x>1$ and $|f(x)|\leq\hat C'x^\beta$ otherwise. Denote $F_k=\{{\mbf y}: p_{\mbf Y}(\mbf y)/\ell_k(\mbf y|\theta)>1\}$, %
    we write
    \begin{align}
\nonumber       k^{d\beta/2}|D_f(\PP_{\mbf Y}||\PP_{\mbf Y|\theta}) | &\leq k^{d\beta/2}\EE_\theta\left[\indic_{F_k}\hat C\frac{p_{\mbf Y}(\mbf y)}{\ell_k(\mbf y|\theta)}\right] + \EE_\theta\left[\indic_{F_k^c}k^{d\beta/2}\hat C'\left(\frac{p_{\mbf Y}(\mbf y)}{\ell_k(\mbf y|\theta)}\right)^\beta \right] \\
                &\leq\hat CK_3 + \hat C'K_2. %
    \end{align}
The domination above makes the following licit:
    \begin{equation}
        \lim_{k\rightarrow\infty} k^{d\beta/2}I_{D_f}(\pi|k)  = \alpha(2\pi)^{d\beta/2}(1-\beta)^{-d/2} \int_{\tilde\Theta}\pi(\theta)^{\beta+1}|\cI(\theta)|^{-\beta/2}d\theta.
    \end{equation}
That concludes the proof of Theorem~\ref{thm:refcompact}.

Finally, to reach Proposition~\ref{prop:jeffcompact}'s results, call $J(\theta)=|\cI(\theta)|^{1/2}/\int_{\tilde\Theta}|\cI(\tilde\theta)|^{1/2}d\tilde\theta$ and write 
    \begin{multline}
        \lim_{k\rightarrow\infty} k^{d\beta/2}(I_{D_f}(J|k)-I_{D_f}(\pi|k)) = \\
            -\alpha (2\pi)^{d\beta/2}(1 -\beta)^{-d/2} \int_{\tilde\Theta}\pi(\theta)^\beta|\cI(\theta)|^{-\beta/2} \pi(\theta)d\theta\\
            + \alpha(2\pi)^{d\beta/2}(1 -\beta)^{-d/2} \left(\int_{\tilde\Theta} |\cI(\theta)|^{1/2}d\theta\right)^{-\beta} .
    \end{multline}
Notice that $\alpha$ being negative implies the concavity of $x\mapsto \alpha x^{-\beta}$, leading to
    \begin{align}
    \nonumber
        &\alpha \int_{\tilde\Theta}\pi(\theta)^\beta |\cI(\theta)|^{-\beta/2}\pi(\theta)d\theta \\
            &\leq \alpha\left(\int_{\tilde\Theta}\pi(\theta)^{-1}|\cI(\theta)|^{1/2}\pi(\theta)d\theta\right)^{-\beta}
            = \alpha \left(\int_{\tilde\Theta} |\cI(\theta)|^{1/2} d\theta\right)^{-\beta} \label{eq:concavinequality}
    \end{align}
and we get $\limtext{k\rightarrow\infty} k^{d\beta/2}(I_{D_f}(J|k)-I_{D_f}(\pi|k)) \geq 0$. Also, as $\alpha<0,\,\beta<1$, the concavity of $x\mapsto\alpha x^{-1/\beta}$ is strict. Therefore, the inequality in equation (\ref{eq:concavinequality}) is an equality if and only if $\pi(\theta)^\beta|\cI(\theta)|^{-\beta/2}$ is constant with respect to  $\theta$, i.e. $\pi=J$.

\end{proof}

\begin{proof}[Proof of Lemma \ref{lem:technicalLemma}]
    We write 
    \begin{align}
    \nonumber
 \EE_\theta\left[\left( \frac{p_{\mbf Y}(\mbf y)}{\ell_k(\mbf y |\theta)} \right)^\zeta\right] 
            &= \EE_\theta \left(\int_{\tilde\Theta} e^{(\tilde\theta-\theta)^T\sqrt{k}S_k} e^{-\frac{k}{2}(\tilde\theta-\theta)^T\hat\cI_k(\hat\theta)(\tilde\theta-\theta)}\pi(\tilde\theta)d\tilde\theta \right)^\zeta \\
\nonumber            &\leq\EE_\theta k^{-d\zeta/2}\left(\int_{\RR^d}  e^{u^TS_k} e^{-\frac{1}{2}u^T\hat\cI_k(\hat\theta)u} du \sup_{\tilde\Theta}|\pi| \right)^\zeta \\
   &\leq\EE_\theta k^{-d\zeta/2}\left(\int_{\RR^d} e^{u^TS_k} e^{-\frac{\|u\|^2}{2}\frac{1}{k}\sum_{i=1}^k\inf_{\tilde\Theta} L^{y_i}} du \sup_{\tilde\Theta}|\pi| \right)^\zeta
    \end{align}
with $\hat\theta$ being a point in $\RR^d$ within the segment between $\theta$ and $\tilde\theta=u/\sqrt{k}+\theta$; with $L^y(\theta)$ being the smallest eigenvalue of $(\nabla^2_\theta\log\ell(y|\theta) )^{1/2}$.
Therefore, $\inf_{\tilde\Theta} L^{y_i}>m$ $\PP_\theta$ a.s. and 
    \begin{align}
        \EE_\theta\left[\left( \frac{p_{\mbf Y}(\mbf y)}{\ell_k(\mbf y |\theta)} \right)^\zeta\right] 
            &\leq k^{-d\zeta/2}K''\EE_\theta\left[m^{-d\zeta/2}e^{\frac{\zeta }{2m}\|S_k\|^2}\right]
    \end{align}    
for some constant $K''>0$. {By considering the Gaussian tail of $S_k$, the proof of the Lemma is complete.}
\end{proof}

\begin{proof}[Proof of Lemma \ref{lem:lem2}]

Denote $E_k=\{p_{\mbf Y}(\mbf y)/\ell_k(\mbf y|\theta)>\rho^{-1/\beta}\}$.
Notice
    \begin{equation}
        \EE_\theta\left[ \indic_{E_k} \frac{p_{\mbf Y}(\mbf y)}{\ell_k(\mbf y|\theta)} \right] = \int_{\cY^k} \indic_{E_k} d\PP_{\mbf Y}(\mbf y).
    \end{equation}
According to Markov's inequality,
    \begin{align}
\nonumber        \EE_\theta\left[ \indic_{E_k} \frac{p_{\mbf Y}(\mbf y)}{\ell_k(\mbf y|\theta)} \right]
            &\leq \rho^{\tilde\beta/\beta} \int_{\cY^k}\left( \frac{p_{\mbf Y}(\mbf y)}{\ell_k(\mbf y|\theta)} \right)^{\tilde\beta} p_{\mbf Y}(\mbf y)d\mu^{\otimes k}(\mbf y) \\
\nonumber
&\leq \rho^{\tilde\beta/\beta} \int_{\cY^k} \left( \frac{p_{\mbf Y}(\mbf y)}{\ell_k(\mbf y|\theta)} \right)^{\tilde\beta} \int_{\tilde\Theta}\ell_k(\mbf y|\tilde\theta)\pi(\tilde\theta) d\tilde\theta d\mu^{\otimes k}(\mbf y)\\
            &\leq \rho^{\tilde\beta/\beta} \int_{\tilde\Theta} \EE_{\tilde\theta}\left[ \left( \frac{p_{\mbf Y}(\mbf y)}{\ell_k(\mbf y|\tilde\theta)} \right)^{\tilde\beta} \left( \frac{\ell_k(\mbf y|\tilde\theta)}{\ell_k(\mbf y|\theta)} \right)^{\tilde\beta} \right] \pi(\tilde\theta)d\tilde\theta  ,
    \end{align}
for some $\tilde\beta>0$ that we determine later. 
{Using the expansion of the likelihood quotient expressed in equation (\ref{eq:DLlogfraclik})} we can write, for some $\hat\theta$ in the segment between $\theta$ and $\tilde\theta$:
    \begin{align}
    \nonumber
         \EE_\theta\left[ \indic_{E_k} \frac{p_{\mbf Y}(\mbf y)}{\ell_k(\mbf y|\theta)} \right]
             &\leq \rho^{\tilde\beta/\beta}\int_{\tilde\Theta} \EE_{\tilde\theta}\left[ \left( \frac{p_{\mbf Y}(\mbf y)}{\ell_k(\mbf y|\tilde\theta)} \right)^{\tilde\beta}  e^{{\tilde\beta}(\tilde\theta-\theta)^T\sqrt{k}S_k} e^{-\frac{k}{2}{\tilde\beta}(\tilde\theta-\theta)^T\hat\cI_k(\hat\theta)(\tilde\theta-\theta)} \right] \pi(\tilde\theta) d\tilde\theta  \\
        \nonumber
            &\leq \rho^{\tilde\beta/\beta} \left(\int_{\tilde\Theta} \EE_{\tilde\theta} \left[ \left( \frac{p_{\mbf Y}(\mbf y)}{\ell_k(\mbf y|\tilde\theta)} \right)^{p{\tilde\beta}}\right] \pi(\tilde\theta)d\tilde\theta \right)^{1/p}\\
           & \quad\times\left(\int_{\tilde\Theta} \EE_{\tilde\theta}\left[   e^{q{\tilde\beta}(\tilde\theta-\theta)^T\sqrt{k}S_k} e^{-q\frac{k}{2}{\tilde\beta}(\tilde\theta-\theta)^T\hat\cI_k(\hat\theta)(\tilde\theta-\theta)} \right] \pi(\tilde\theta) d\tilde\theta\right)^{1/q},
    \end{align}
the second inequality being a consequence of Hölder's inequality (necessarily $1/p+1/q=1$).
The first integral in the right-hand side of the equation verifies
    \begin{align}
        \int_{\tilde\Theta} \EE_{\tilde\theta} \left[ \left( \frac{p_{\mbf Y}(\mbf y)}{\ell_k(\mbf y|\tilde\theta)} \right)^{p{\tilde\beta}}\right] \pi(\tilde\theta)d\tilde\theta \leq K_2 k^{-p{\tilde\beta} d/2}  
    \end{align}
when ${\tilde\beta} p/2<\xi$ according to Lemma \ref{lem:technicalLemma}.

Now, we look for an upper-bound of the second integral:
    \begin{equation}
  I = \int_{\tilde\Theta}\EE_{\tilde\theta}\left[ e^{q\tilde\beta(\tilde\theta-\theta)^T\sqrt{k}S_k} e^{-q\frac{k}{2}\tilde\beta(\tilde\theta-\theta)^T\hat\cI_k(\hat\theta)(\tilde\theta-\theta)}\right]\pi(\tilde\theta)d\tilde\theta .
    \end{equation}
Beforehand, let $\tilde\theta\in\tilde\Theta$,  $t\in\RR$ and $v\in\RR^d$ with $\|v\|=1$. For any $a >0$, we have 
$|tv^TS_k|\leq \frac{a}{2} t^2 +\frac{1}{2a} |v^TS_k|^2 \leq  \frac{a}{2} t^2 +\frac{1}{2a} \|S_k\|^2 $.
Choosing $a=m/(2\xi)$, we get
\begin{equation}
    \EE_\theta e^{tv^TS_k} \leq 
e^{{m t^2}/(4\xi)} \EE_{\tilde\theta} e^{ ({\xi}/{m}) \|S_k\|^2} \leq K_1 e^{ {t^2c^2}/{2}}
\end{equation}
with $c^2=\xi/(2m)$.
Thus, for any $\tilde\theta\in\tilde\Theta$:
    \begin{equation}
        \EE_{\tilde\theta}\left[ e^{q\tilde\beta(\tilde\theta-\theta)^T\sqrt{k}S_k} e^{-q\frac{k}{2}\tilde\beta(\tilde\theta-\theta)^T\hat\cI_k(\hat\theta)(\tilde\theta-\theta)}\right] \leq K_1 e^{\frac{c^2}{2}q^2\tilde\beta^2k\|\tilde\theta-\theta\|^2} e^{-q\frac{k}{2}\tilde\beta m\|\tilde\theta-\theta\|^2}.
    \end{equation}
Therefore 
\begin{align}
\nonumber
  I
    &\leq K_1k^{-d/2}\int_{\RR^d} e^{\frac{c^2}{2}q^2\tilde\beta^2\|x\|^2} e^{-\frac{q}{2}\tilde\beta m\|x\|^2} \sup_{\tilde\Theta}\pi dx\\
       &\leq K_1k^{-d/2}\int_{\RR^d} e^{\frac{m}{4\xi}q^2\tilde\beta^2\|x\|^2} e^{-\frac{q}{2}\tilde\beta m\|x\|^2} \sup_{\tilde\Theta}\pi dx.
\end{align}
Let us suppose the integral of the right hand term above to be a finite constant that we denote $L$; for this to be true, the inequality $q\tilde\beta<2\xi$ must be satisfied.
In that case, one gets 
\begin{equation}
     \EE_\theta\left[ \indic_{E_k} \frac{p_{\mbf Y}(\mbf y)}{\ell_k(\mbf y|\theta)} \right] \leq \rho^{\tilde\beta/\beta} K_2^{1/p}L^{1/q}K_1^{1/q} k^{-d(\tilde\beta+1/q)/2}. %
\end{equation}

Choose now $\tilde\beta=\beta/2$, and $q=p=2$. Thus $p\tilde\beta/2<\xi$ and $q\tilde\beta/2<\xi$ as required. Moreover, $\tilde\beta+1/q>\beta$ given $1>\beta$. This way, %
    \begin{equation}
        \EE_\theta\left[ \indic_{E_k} \frac{p_{\mbf Y}(\mbf y)}{\ell_k(\mbf y|\theta)} \right] \leq \rho^{1/2} K_3 k^{-d\beta/2},
    \end{equation}
with $K_3 = K_2^{1/2}L^{1/2}K_1^{1/2}$.

\end{proof}

\begin{proof}[Proof of Theorem \ref{thm:mainfinal}]
    Theorem \ref{thm:refcompact} and Proposition \ref{prop:jeffcompact}  {coupled with remark \ref{rem:overlinef}} ensure the Jeffreys prior to be the  {$D_{\overline f}$-reference prior} by choosing any sequence of compact subsets $(\Theta_i)_{i\in I}$ matching the definition.
    To prove its uniqueness, we suppose to have a second reference prior $\pi^\ast\in\cP$ associated with another sequence of compact subsets $(\tilde\Theta_h)_{h\in H}$. Therefore, the equality must be verified within equation (\ref{eq:thmI-Jgeqpi}) for any restriction of $J$ and $\pi^\ast$ to one of the $\tilde\Theta_h$, making them to be equal. Eventually, as $\bigcup_{h\in H}\tilde\Theta_h=\Theta$, we obtain $J=\pi^\ast$, hence the result.
\end{proof}

\section{Conclusion and prospects}\label{sec:conclusion}

Our work contributes to the enrichment of reference prior theory on multiple fronts.
First and foremost, we establish a connection between reference priors and Global Sensitivity Analysis, {shedding light on the interpretation of the mutual information as a sensitivity measure}. Second, by generalizing the definition of mutual information, we propose a framework that suggests a multitude of reference priors corresponding to various dissimilarity measures. {Finally, we provide  an analytical criterion that aids in their definition.}

We interpret reference prior criteria as sensitivity impacts.
That interpretation aligns with purposes of statistical studies that are based on real data.
This perspective enhances trustworthiness of their selection in practice.
In those application contexts, %
{pushing further the robustness of reference priors that are constructed from Jeffreys ---or hierarchical versions of Jeffreys--- supports the objectivity of its choice }
and  of its resulting estimates.
Despite potential computational challenges in certain scenarios, the substantial advantages offered by the objectivity of this method cannot be overlooked.
Also, existing studies have already made strides in handling the numerical computation aspect \cite{Nalisnick2017, GauchyJds2023}. Future works could target such computation on the basis of the generalized mutual information we suggest.

Looking ahead, future research could explore ways such as the investigation of other generalized reference priors beyond those based on $f$-divergences. Those could be derived considering alternative dissimilarity measures inspired by GSA, for instance; or imposing judicious constraints on less general classes of distributions. {Such endeavors could harness the asymptotic expression of generalized mutual information}, catering to reliability applications that prioritize auditability and robustness \cite{Bousquet2023}. These directions hold promise for further advancing the utility and applicability of reference prior theory in Bayesian analysis.

\begin{appendix}
\section{A heuristic for understanding the main result}\label{app:heuristic}

In this Appendix, we present the main heuristic of our proof which leads to the Jeffreys prior as a maximizer of the $D_f$-mutual information.
It is similar to the heuristics of \citet{Clarke1994} or \citet[\S 2]{Mure2018} who derived the reference prior for the original mutual information. It is initiated by the asymptotic expansion
\begin{equation}
    \log\left(\frac{\prod_{i=1}^k \ell(y_i|\tilde\theta)}{\prod_{i=1}^k \ell(y_i|\theta)}\right) = (\tilde\theta-\theta)^T\sum_{i=1}^k\nabla_\theta\log \ell(y_i|\theta) +\frac{1}{2}(\tilde\theta-\theta)^T\nabla_\theta^2\log \ell(y_i|\theta')(\tilde\theta-\theta)
\end{equation}
for every $\theta, \tilde\theta$, and for a $\theta'$ in the segment between $\theta$ and $\tilde\theta$.

We remark that $p_{\mbf Y}(\mbf y)/\ell_k(\mbf y|\theta) = \int_\Theta \frac{\prod_{i=1}^k \ell(y_i|\tilde\theta)}{\prod_{i=1}^k \ell(y_i|\theta)} \pi(\tilde\theta)d\tilde\theta$.
Following the same lines as in the proof of Proposition \ref{prop:cvproba}, we can show that $\frac{1}{k}\sum_{i=1}^k\nabla_\theta^2\log \ell(y_i|\theta')$ is close to $\cI(\theta')$ and we can intuit the concentration of the mass of the integrated function around the neighborhood of $\theta$. This leads to the approximation
\begin{align}
    \frac{p_{\mbf Y}(\mbf y)}{\ell_k(\mbf y|\theta)} 
        &\simeq\pi(\theta)(2\pi)^{d/2}|k\cI(\theta)|^{-1/2}\exp\left(\frac{1}{2k}\nabla_\theta\log\ell_k(\mbf y|\theta)^T\cI(\theta)^{-1}\nabla_\theta\log\ell_k(\mbf y|\theta) \right).
\end{align}

Therefore, under appropriate assumptions, it can be anticipated that
\begin{multline}
    I_{D_f}(\pi|k) \\ \simeq \int_\Theta\int_{\cY^k}f\left[\pi(\theta)(2\pi)^{d/2}|k\cI(\theta)|^{-1/2}\exp\left(\frac{1}{2k}\nabla_\theta\log\ell_k(\mbf y|\theta)^T\cI(\theta)^{-1}\nabla_\theta\log\ell_k(\mbf 
 y|\theta)\right)\right] \\
    \cdot \ell_k(\mbf y|\theta)\pi(\theta)d\mu^{\otimes k}(\mbf y)d\theta.
\end{multline}
Using that $k^{-1/2}\sum_{i=1}^k\nabla_\theta\log\ell(y_i|\theta)$ is asymptotically normal with mean $0$ and variance $\cI(\theta)$ (for $\mbf y$ sampled from $\PP_{\mbf Y|\theta}$), we get:
    \begin{align}
        I_{D_f}(\pi|k) \simeq \int_\Theta\int_{\RR^k} f\left(\pi(\theta)(2\pi)^{d/2}|k\cI(\theta)|^{-1/2}e^{\frac{\|x\|^2}{2}}\right) e^{-\frac{\|x\|^2}{2}} (2\pi)^{-d/2}dx\pi(\theta)d\theta.
    \end{align}
We denote by $\hat I_{D_f}(\pi|k)$ the right hand term.
We use the concavity of $x\mapsto f(1/x)$ to obtain
    \begin{align}
        \hat I_{D_f}(\pi|k) &\leq \int_{\RR^k}f\left(\left(\int_\Theta (2\pi)^{-d/2}|k\cI(\theta)|^{1/2} e^{-\frac{\|x\|^2}{2}}d\theta\right)^{-1}\right)e^{-\frac{\|x\|^2}{2}}(2\pi)^{-d/2}dx %
    \end{align}
and on another hand, denoting $J(\theta) = |k\cI(\theta)|^{1/2}\left(\int_\Theta|k\cI(\theta)|^{1/2}d\theta\right)^{-1}$, we find
    \begin{align}
        \hat I_{D_f}(J|k) = \int_{\RR^k}f\left(\left(\int_\Theta (2\pi)^{-d/2}|k\cI(\theta)|^{1/2} e^{-\frac{\|x\|^2}{2}}d\theta\right)^{-1}\right)e^{-\frac{\|x\|^2}{2}}(2\pi)^{-d/2}dx
    \end{align}
from which we deduce that the Jeffreys prior maximizes asymptotically the $D_f$-mutual information.

\section{Additional elements when $\beta<0$}\label{app:betaleq0}

In this appendix, we propose an extension of the preliminary results obtained in \cite{VanBiesbroeckJDS2023}, in which similar $f$-divergences were considered. The framework in \cite{VanBiesbroeckJDS2023} and in this appendix is, however, different from the one in this paper.
In \cite{VanBiesbroeckJDS2023} and in this appendix, $\beta$ is assumed to belong to $(-1,0)$.

\begin{assu}
\label{assu:JDS}
    There exist $\tau>0$ and $\delta>0$ such that the quantity 
        \begin{equation}
            \EE_\theta\left[\exp\Big(\tau\sup_{\|\tilde\theta-\theta\|<\delta}\|\nabla^2_\theta\log\ell(y|\tilde\theta)\|\Big)\right] 
        \end{equation}
    is continuous with respect to  $\theta$. %
\end{assu}

\begin{thm}\label{thm:JDS}
    Suppose Assumptions \labelcref{assu:lgolikelihood,assu:fisher,assu:JDS}. Assume that $f:(0,+\infty)\to\RR$ is measurable, locally bounded and satisfies the two following conditions:
\begin{align}        f(x)&\aseq{x\rightarrow0^+}  \alpha x^\beta +o(x^\beta) ,\\
        f(x)&\aseq{x\rightarrow +\infty}O(x) ,
    \end{align}
for some $\alpha$, $\beta$, such that $\beta<0$. %
For any prior $\pi$ positive on $\tilde\Theta$ and absolutely continuous with respect to  the Lebesgue measure with continuous and positive Radon–Nikodym derivative denoted by $\pi$ as well, the quantity $k^{d\beta/2}I_{D_f}(\pi|k)$ has a positive limit when $k\to\infty$:
\begin{align}
\label{eq:limitkbeta}
        \lim_{k\rightarrow\infty} k^{d\beta/2} I_{D_f}(\pi|k) = 
\alpha C_\beta \int_{\tilde\Theta}\pi(\theta)^{1+\beta} |\cI(\theta)|^{-\beta/2}  d\theta ,
    \end{align}
where $ C_\beta = (2\pi)^{d\beta/2} (1-\beta)^{-d/2}$. %
Moreover, if $\alpha(\beta+1)>0$, then
    \begin{equation}
        \lim_{k\rightarrow\infty}k^{d\beta/2}(I_{D_f}(J|k)-I_{D_f}(\pi|k))\geq 0 ,
    \end{equation}
where $J(\theta)=|\cI(\theta)|^{1/2}/\int_{\tilde\Theta}|\cI(\tilde\theta)|^{1/2}d\tilde\theta$ denotes the Jeffreys prior. The equality stands if and only if 
{$\pi=J$}.
\end{thm}

Here, we provide the following new result:
\begin{thm}
    Consider the assumptions of Theorem \ref{thm:JDS} and suppose that Assumptions \labelcref{assu:infeighes,assu:gausstailSk} stand for every compact subset $\tilde\Theta$ of $\Theta$. %
    The Jeffreys prior is a $D_f$-reference prior over class $\cP$ with rate $k^{d\beta/2}$. It is the unique reference prior over the class $\cP^\ast=\{\pi\in\cP,\,\pi>0\}$.
\end{thm}

\begin{proof}
    Consider a subset $\Theta\subset\RR^d$ and $\pi$ a prior with continuous density on $\Theta$. %

Let $\bigcup_{i\in I} K_i=\RR^d$  be a a compact recovering of the space and denote $\Theta_{i,j}=K_i\cap(\{\pi\geq 1/j\}\cup\{\pi=0\})$ for any $i\in I,\,j\in\NN^\ast$. Those are compact subsets of $\Theta$ such that $\bigcup_{i,j}\Theta_{i,j}=\Theta$.

Considering $\pi_{i,j}$ the restriction (with re-normalization) of the prior $\pi$ to $\Theta_{i,j}$, notice that $I_{D_f}(\pi_{i,j}|k)=I_{D_f}(\hat\pi_{i,j}|k)$ with $\hat\pi_{i,j}$ being the restriction of $\pi$ to $\hat\Theta_{i,j}=K_i\cap\{\pi\geq 1/j\}$. As $\hat\pi_{i,j}$ is a positive and continuous prior on the compact $\hat\Theta_{i,j}$ one may apply Theorem \ref{thm:JDS} to obtain
\begin{equation}
    \lim_{k\rightarrow\infty} k^{d\beta/2}(I_{D_f}(\hat J_{i,j}|k)-I_{D_f}(\hat\pi_{i,j})|k)\geq 0
\end{equation}
Remark
\begin{multline}
    \lim_{k\rightarrow\infty} k^{d\beta/2}I_{D_f}(\hat J_{i,j}|k) = C_\beta\alpha\left(\int_{\hat\Theta_{i,j}} |\cI(\theta)|^{1/2}d\theta \right)^{-\beta}\\
        \leq C_\beta\alpha\left(\int_{\Theta_{i,j}} |\cI(\theta)|^{1/2}d\theta \right)^{-\beta} = \lim_{k\rightarrow\infty} k^{d\beta/2}I_{D_f}(J_{i,j}|k)
\end{multline}
as a result of $x\mapsto\alpha x^{-\beta}$ being an increasing function as assumed.
Eventually, we obtain
\begin{equation}
    \lim_{k\rightarrow\infty} k^{d\beta/2}(I_{D_f}(J_{i,j}|k)-I_{D_f}(\pi_{i,j}|k))\geq 0.
\end{equation}
This being true for any couple $(i,j)\in I\times\NN^\ast$, this makes the Jeffreys prior a $D_f$-reference prior over $\cP$. %

Concerning the class $\cP^\ast$, any prior in $\cP^\ast$ satisfies the assumptions of Theorem \ref{thm:JDS} on every compact subset of $\Theta$.
The $D_f$-reference characteristic of Jeffreys prior and its uniqueness are a direct consequence of that statement, analogous to the proof of Theorem \ref{thm:mainfinal}.
\end{proof}

\end{appendix}
\begin{acks}[Acknowledgments]
 The author would like to thank his PhD advisors Professor Josselin Garnier (CMAP, CNRS, École polytechnique, Institut Polytechnique de Paris), Cyril Feau (Université Paris-Saclay, CEA, Service d'Études Mécaniques et Thermiques), and Clément Gauchy  (Université Paris-Saclay, CEA, Service de Génie Logiciel pour la Simulation) for their guidance and support.
\end{acks}

\section*{Conflict of interest}

The author reports there are no competing interests to declare.

\bibliographystyle{i-sart-nameyear} %
\bibliography{bibliography}       %

\end{document}